\documentclass[a4paper, 12 pt, oneside, reqno]{amsart} 
\usepackage{amsfonts, amssymb, amsmath, eucal, amsthm}
\usepackage[all]{xy}
\usepackage{verbatim}
\usepackage{tikz}
\usepackage[hidelinks]{hyperref}
\usepackage{bbm}

\newtheorem{lemma}{Lemma}
\numberwithin{lemma}{section}
\newtheorem{theorem}[lemma]{Theorem}
\newtheorem*{theorem*}{Theorem}
\newtheorem*{corollary*}{Corollary}
\newtheorem{proposition}[lemma]{Proposition}
\newtheorem{corollary}[lemma]{Corollary}

\theoremstyle{definition}
\newtheorem{definition}[lemma]{Definition}
\newtheorem{note}[lemma]{Note}
\newtheorem{example}[lemma]{Example}
\newtheorem{remark}[lemma]{Remark}

\newtheorem*{remark*}{Remark}

\usepackage{geometry}

\newcommand{\MC}{\mathrm{MC}}
\newcommand{\MH}{\mathrm{MH}}
\newcommand{\N}{\mathbb{N}}
\newcommand{\Q}{\mathbb{Q}}
\newcommand{\R}{\mathbb{R}}
\newcommand{\Z}{\mathbb{Z}}
\newcommand{\bfx}{{\underline{x}}}
\newcommand{\bfy}{{\underline{y}}}

\newcommand{\bfi}{\underline{i}}

\DeclareMathOperator{\Hom}{Hom}
\newcommand{\cutprod}{\cdot}
\newcommand{\hahn}{\Q(\!(q^\R)\!)}
\DeclareMathOperator{\rank}{rank}
\newcommand{\adj}{\mathrm{Adj}}
\newcommand{\V}{\mathcal{V}}
\newcommand{\A}{\mathcal{A}}

\newcommand{\Ch}{\mathrm{Ch}}
\newcommand{\Ab}{\mathrm{Ab}}
\newcommand{\two}{\mathbf{2}}
\newcommand{\set}{\mathrm{Set}}

\newcommand{\true}{\mathtt{t}}
\newcommand{\false}{\mathtt{f}}
\newcommand{\AW}{\mathrm{AW}}
\newcommand{\uu}{a}
\newcommand{\vv}{b}

\title{Magnitude cohomology}
\author{Richard Hepworth}
\address{Institute of Mathematics\\
University of Aberdeen\\
Aberdeen AB24 3UE\\
United Kingdom}
\email{r.hepworth@abdn.ac.uk}


\begin{document}

\begin{abstract}
Magnitude homology was introduced by Hepworth and Willerton in the case of
graphs, and was later extended 
by Leinster and Shulman to metric spaces
and enriched categories.
Here we introduce the dual theory, magnitude cohomology,
which we equip with the structure of an associative unital graded ring.
Our first main result is a `recovery theorem'
showing that the magnitude cohomology ring of a finite metric space
completely determines the space itself.
The magnitude cohomology ring is non-commutative in general, for example
when applied to finite metric spaces, but in some settings it is
commutative, for example when applied to ordinary categories.
Our second main result explains this situation by proving that the magnitude
cohomology ring of an enriched category is graded-commutative
whenever the enriching category is cartesian.
We end the paper by giving
complete computations of magnitude cohomology rings for several
large classes of graphs.
\end{abstract}

\maketitle

\section{Introduction}

\subsection{Overview}
In this paper we introduce and investigate \emph{magnitude cohomology}
of generalised metric spaces and enriched categories.
Our theory is dual, in the same sense as singular homology and cohomology, 
to the theory of \emph{magnitude homology} introduced by
Hepworth and Willerton~\cite{HepworthWillerton}
and later extended by Leinster and Shulman~\cite{LeinsterShulman}.
As in the singular case, we find that the introduction of 
cohomology adds strength and structure to the whole theory.
We show that the magnitude cohomology groups
form a unital associative graded ring, 
which is noncommutative in many cases.
We prove a \emph{recovery theorem}
which shows that the magnitude cohomology of a large class of metric
spaces (including finite metric spaces and arbitrary directed graphs)
completely determines the metric space in question.
Moving to the context of enriched categories,
we prove that the magnitude cohomology ring of an enriched category
is graded-commutative so long as the enriching category is cartesian.
Finally, specialising to undirected graphs,
we give complete computations of the magnitude cohomology ring
for diagonal graphs and for odd cyclic graphs, and we establish a
connection between magnitude cohomology and the quiver algebra of a graph.

\subsection{Background and motivation}
Leinster's theory of \emph{magnitude}
is a mechanism that associates numerical invariants
to mathematical objects of various kinds.
In greatest generality it is an invariant of enriched categories.
The power of this theory is that different choices of enriching category
take us into different regions of mathematics, with
notions of magnitude for posets, categories,
graphs, metric spaces and more,
and that the resulting invariants are meaningful and interesting
in many of these settings.
In the case of finite posets the magnitude
is precisely the Euler characteristic of the order complex.
For finite categories the magnitude encompasses the Euler characteristic
of the classifying space and the cardinality of groupoids, 
but is defined more generally.
In the case of graphs the magnitude is a formal power series
with many attractive properties, such as 
product and inclusion-exclusion formulas,
and invariance under certain Whitney twists.
The magnitude of finite metric spaces is a cardinality-like invariant,
described as the `effective number of points', that
first arose as a measurement of biological 
diversity~\cite{SolowPolasky}.
Perhaps more importantly, the magnitude of finite metric spaces 
can be extended to compact metric spaces.
In this setting the magnitude 
is known to encode geometric information,
such as the volume and perimeter of domains in odd-dimensional 
Euclidean space~\cite{GimperleinGoffeng},
although it remains rather mysterious and difficult to compute.
We refer the reader to~\cite{LeinsterEuler} for 
magnitude of ordinary categories and posets,
\cite{LeinsterMagnitudeMetric} for enriched categories and metric spaces,
and~\cite{LeinsterGraph} for graphs.

\emph{Magnitude homology} 
was first introduced by the author and 
Willerton~\cite{HepworthWillerton} in the setting of graphs, 
and it categorifies the magnitude in exactly
the same sense that Khovanov homology categorifies the Jones polynomial.
Leinster and Shulman later extended magnitude
homology to categories enriched in a 
semicartesian category~\cite{LeinsterShulman},
they showed that it determines the magnitude in favourable circumstances,
and they specialised the definition to obtain
the magnitude homology of metric spaces.
Magnitude homology has shown itself
to be an important extension and refinement of magnitude,
with many characteristic features of homology theories
and categorification:
\begin{itemize}
	\item
	Magnitude homology of graphs has properties
	(such as K\"unneth and Mayer-Vietoris theorems) that categorify
	and explain properties of magnitude (such as the product rule 
	and inclusion-exclusion formula)~\cite{HepworthWillerton}.
	\item
	The phenomenon of alternating coefficients in the magnitude
	of certain graphs was explained by the notion of
	{diagonal graphs} in~\cite[Section~7]{HepworthWillerton}.
	\item
	Graphs (and therefore metric spaces)
	with the same magnitude can have distinct
	magnitude homology groups.  The first known example of such a pair
	is the $4\times 4$ rook's graph 
	and the Shrikhande graph,
	and is due to Yuzhou Gu.  See Appendix~A of Gu's paper~\cite{Gu}
	and also the comments at the blogpost~\cite{Willerton}.
	\item
	Magnitude homology of graphs (and therefore metric spaces)
	can contain torsion, and in particular magnitude
	homology is not determined by its ranks.
	The first example of this was obtained by Kaneta and
	Yoshinaga in~\cite[Corollary~5.12]{KanetaYoshinaga}.
	\item
	A metric space $X$ is
	Menger convex if and only if its magnitude homology groups
	vanish in (homological) 
	degree $1$~\cite[Corollary 4.5]{LeinsterShulman}.
\end{itemize}
We would also like to mention recent work of Otter~\cite{otter},
which establishes a connection between magnitude homology and topology.
Otter introduces \emph{blurred magnitude homology} of metric spaces, 
a persistent version of the theory, and shows that a certain inverse
limit of the blurred magnitude homology produces the \emph{Vietoris homology}.
The latter is a homology theory for metric spaces
that coincides with singular homology in certain cases,
for example for compact Riemannian manifolds.

\subsection{Metric spaces}
Our first results are in the setting of 
\emph{generalised metric spaces}~\cite{Lawvere,LawvereReprint}. 
We introduce the magnitude cohomology of a generalised metric space $X$,
which is a bigraded abelian group $\MH^\ast_\ast(X)$
consisting of groups $\MH_\ell^k(X)$ where $k=0,1,2,\ldots$ 
and $\ell\in[0,\infty)$.
It is equipped with a product operation
\[
	\MH^j_\ell(X)\otimes \MH^k_m(X)
	\longrightarrow
	\MH^{j+k}_{\ell+m}(X)
\]
that gives it the structure of an associative unital ring.
This ring structure is in general noncommutative,
as we show in Proposition~\ref{proposition-noncommutative}.
In Theorem~\ref{theorem-recovery} we give a 
`recovery theorem' which shows that for a large class of metric
spaces $X$, including all graphs and finite metric spaces, 
the magnitude cohomology ring $\MH^\ast_\ast(X)$
is sufficient to determine $X$ precisely.
This is in stark contrast to the situation for magnitude homology, 
or for magnitude cohomology without the ring structure, 
where for example any two trees with
the same number of vertices have isomorphic magnitude homology.

\subsection{Enriched categories}
The next part of the paper deals with magnitude cohomology
of enriched categories.  For this we fix a symmetric monoidal
semicartesian category $\V$ and a strong monoidal functor
$\Sigma\colon\V\to\A$ into a closed symmetric monoidal
abelian category $\A$.  In this situation 
Leinster and Shulman defined the \emph{magnitude homology}
$H^\Sigma_\ast(X)$ of a $\V$-category $X$~\cite{LeinsterShulman}.
We define the dual theory, 
the \emph{magnitude cohomology} $H^\ast_\Sigma(X)$ of 
a $\V$-category $X$, and we equip it with 
a product
\[
	H^j_\Sigma(X)\otimes H^k_\Sigma(X)\longrightarrow H^{j+k}_\Sigma(X)
\]
and unit $1_\A\to H^0_\Sigma(X)$ that make it into an associative
unital graded ring in $\A$.
By choosing appropriate $\V$, $\A$ and $\Sigma$,
we obtain magnitude cohomology rings for posets,
small categories, and generalised metric spaces,
and we show that these respectively
recover the cohomology of the order complex
with the cup product, the cohomology of the classifying space
with the cup product, and the magnitude cohomology ring
respectively.
Note that in the first two cases the product is graded-commutative,
but in the third case it is not.  
In Theorem~\ref{theorem-commutativity}
we explain this phenomenon by showing that in general, 
when the enriching category $\V$ is cartesian,
the magnitude cohomology ring is graded-commutative.  
The enriching categories for posets and small categories
\emph{are} cartesian, while the enriching category for
generalised metric spaces is \emph{not}.

\subsection{Graphs}
Finally we specialise to finite graphs,
and we give complete computations of the magnitude cohomology ring
for several classes of graphs.
The magnitude cohomology rings in question are all highly nontrivial,
but nevertheless they all admit nice presentations.

In Theorem~\ref{theorem-diagonal} we identify
the \emph{diagonal part} of the magnitude cohomology ring,
given by the groups $\MH^k_k(G)$, as a quotient of the path algebra 
of the quiver obtained from $G$ by doubling the edges.
In~\cite[Section~7]{HepworthWillerton}
we introduced \emph{diagonal graphs}, which are 
graphs whose magnitude homology (and therefore cohomology)
is concentrated on this diagonal,
and we identified various large classes of diagonal graphs.
Work of Gu~\cite{Gu} has added to the known examples.
Theorem~\ref{theorem-diagonal} therefore 
gives a complete description of the magnitude
cohomology ring of any diagonal graph, and we make this
description explicit in the 
case of trees, complete graphs, and complete bipartite graphs.

In Theorem~\ref{theorem-cyclic} we give an explicit presentation 
of the magnitude cohomology rings of the odd cyclic graphs
(which are not diagonal).
This is based on Gu's computation of the magnitude homology
of odd cyclic graphs given in Theorem~4.6 of~\cite{Gu}.

\subsection{Organisation of the paper}
The paper is organised as follows.
In section~\ref{section-magcoh-metric} we define the magnitude
cohomology of metric spaces, together with its ring structure,
and we relate it to magnitude. 
In section~\ref{section-recovery} we state and prove
our recovery theorem.
In section~\ref{section-enriched-homology} we give a brief
exposition of Leinster and Shulman's theory of magnitude homology
of enriched categories. Then in section~\ref{section-enriched-cohomology}
we define magnitude cohomology of enriched categories,
we define its ring structure, and we prove that
it is commutative when the enriching category is cartesian.
In section~\ref{section-graph} we explore the case of finite graphs,
we express the diagonal part of magnitude cohomology as a quotient
of the path algebra of the associated quiver, and we compute
the magnitude cohomology ring for several classes of diagonal graphs.
In section~\ref{section-cyclic} we compute the magnitude cohomology
ring of odd cyclic graphs.

\subsection{Acknowledgements}
Thanks to Simon Willerton for some useful comments, and to Tom Leinster
for explaining the connection between dagger categories and symmetry
in generalised metric spaces.

\section{Magnitude cohomology of metric spaces}
\label{section-magcoh-metric}

In this section we recall the magnitude homology of generalised metric spaces,
and we then define their magnitude cohomology
and equip it with the structure of a graded associative unital ring.
We show that this ring is typically not graded commutative,
and we show how to recover the magnitude from the magnitude cohomology.

We work with {generalised metric spaces} in the sense
of Lawvere~\cite{Lawvere,LawvereReprint}.
Recall that a \emph{generalised metric space}, 
or \emph{extended pseudo-quasi-metric space}, 
is defined in the same
way as a metric space, except that the metric takes values
in $[0,\infty]$ (extended), distances between distinct points
may be $0$ (pseudo), and the distance from $a$ to $b$
need not equal the distance from $b$ to $a$ (quasi).
Generalised metric spaces are the same thing as categories
enriched in $[0,\infty]$. The objects of the enriched category
correspond to the points of the space, and the morphism
object $X(a,b)\in[0,\infty]$ corresponds to the distance from $a$ to $b$.
We will work with generalised metric spaces where possible,
restricting to extended quasi-metric spaces where necessary.

\begin{definition}[Magnitude homology]
\label{definition-magnitude-homology}
We recall the definition of magnitude homology
from sections~2 and~3 of~\cite{HepworthWillerton}
(for graphs) 
and section~3 of~\cite{LeinsterShulman}
(for arbitrary metric spaces).

Let $X$ be a generalised metric space.
A $k$-\emph{simplex} or just \emph{simplex} 
in $X$ is a tuple $(x_0,\ldots,x_k)$
of elements of $X$ in which consecutive entries are distinct,
i.e.~$x_0\neq x_1\neq\cdots\neq x_k$.
The \emph{degree} of the simplex $(x_0,\ldots,x_k)$ is $k$,
and its \emph{length} is
\[
	\ell(x_0,\ldots,x_k)
	=d(x_0,x_1)+\cdots+d(x_{k-1},x_k).
\]
The \emph{magnitude chain complex} of $X$, denoted $\MC_{\ast,\ast}(X)$,
is the chain complex of $\R$-graded abelian groups 
defined as follows.
The $k$-chains in degree $\ell$, denoted $\MC_{k,\ell}(X)$,
is defined to be the free abelian group on the simplices
of degree $k$ and length $\ell$:
\[
	\MC_{k,\ell}(X)
	=
	\Z\left\{
		(x_0,\ldots,x_k)\in X^{k+1}
		\mid	
		\ell(x_0,\ldots,x_k)=\ell,\ 
		x_0\neq x_1\neq \cdots\neq x_k
	\right\}
\]
The differential
\[
	\partial\colon\MC_{k,l}(X)\longrightarrow \MC_{k-1,l}(X)
\]
is defined by
\[
	\partial=-\partial_1+\partial_2-\cdots+(-1)^{k-1}\partial_{k-1}
\]
where
\[
	\partial_i(x_0,\ldots,x_k)=
	\left\{
	\begin{array}{ll}
		(x_0,\ldots,\widehat{x_i},\ldots,x_k)
		& \text{if }d(x_{i-1},x_{i+1})=d(x_{i-1},x_i)+d(x_i,x_{i+1}),
		\\
		0
		&
		\text{otherwise}.
	\end{array}
	\right.
\]
The \emph{magnitude homology} $\MH_{\ast,\ast}(X)$
of $X$ is the homology of the magnitude chain complex:
\[
	\MH_{\ast,\ell}(X)=H_\ast(\MC_{\ast,\ell}(X)).
\]
Note that in~\cite{LeinsterShulman} $\MH_{k,\ell}(X)$
is denoted by $H_{k,\ell}(X)$.
If $f\colon X\to Y$ is a map of generalised metric spaces
that does not increase distances, i.e.~$d_Y(f(x),f(x'))\leqslant
d_X(x,x')$ for all $x,x'\in X$,
then the \emph{induced chain map}
\[ f_\#\colon\MC_{\ast,\ast}(X)\to \MC_{\ast,\ast}(Y)\]
is defined by
\[
	f_\#(x_0,\ldots,x_k)
	=
	\left\{\begin{array}{ll}
		(f(x_0),\ldots,f(x_k))
		&
		\text{if }\ell(f(x_0),\ldots,f(x_k))=\ell(x_0,\ldots,x_k)
		\\
		0
		&
		\text{otherwise}.
	\end{array}\right.
\]
The \emph{induced map in homology} is the map
\[
	f_\ast\colon\MH_{\ast,\ast}(X)\longrightarrow\MH_{\ast,\ast}(Y)
\]
obtained from $f_\#$.
\end{definition}

We now define magnitude \emph{co}homology
by dualising the above definition and equipping
it with a product structure.

\begin{definition}[Magnitude cohomology]
\label{definition-magcoh}
Let $X$ be a generalised metric space.
The \emph{magnitude cochain complex of $X$} is the 
dual to the magnitude chain complex:
\[
	\MC_\ell^\ast(X)=\Hom(\MC_{\ast,\ell}(X),\Z).
\]
The \emph{magnitude cohomology} of $X$, denoted $\MH_\ast^\ast(X)$,
is the cohomology of the magnitude cochains:
\[
	\MH_\ell^\ast(X) = H^\ast(\MC_\ell^\ast(X)).
\]
If $f\colon X\to Y$ is a non-increasing map of 
generalised metric spaces, then the dual of
$f_\#$ gives the \emph{induced cochain map}
\[f^\#\colon\MC_\ast^\ast(X)\longrightarrow \MC_\ast^\ast(Y)\]
and the \emph{induced map in cohomology}
\[f^\ast\colon\MH_\ast^\ast(X)\longrightarrow\MH_\ast^\ast(Y).\]
Given $\varphi\in\MC^{k_1}_{\ell_1}(X)$
and $\psi\in\MC^{k_2}_{\ell_2}(X)$,
the \emph{product} 
$\varphi\cutprod\psi \in \MC^{k_1+k_2}_{\ell_1+\ell_2}(X)$
is defined on
$(x_0,\ldots,x_{k_1+k_2})\in \MC_{k_1+k_2,\ell_1+\ell_2}(X)$
by 
\[
	(\varphi\cutprod\psi) (x_0,\ldots,x_{k_1+k_2})
	=
	\varphi(x_0,\ldots,x_{k_1})\cdot \psi(x_{k_1},\ldots,x_{k_1+k_2})
\]
if the lengths are compatible in the sense that 
$\ell(x_0,\ldots,x_{k_1})=\ell_1$ and
$\ell(x_{k_1},\ldots,x_{k_1+k_2})=\ell_2$,
and by 
\[
	(\varphi\cutprod\psi) (x_0,\ldots,x_{k_1+k_2})
	=
	0
\]
otherwise.
The \emph{unit} $u\in\MC^0_0(X)$ is the cochain defined by the rule
\[
	u(x_0) = 1
\]
for every $0$-simplex $(x_0)$.
The reader may readily verify that the product
is strictly associative and unital, with unit $u$,
and that it satisfies the Leibniz rule
\[
	\partial^\ast(\varphi\cutprod\psi)
	=\partial^\ast\varphi\cutprod\psi 
	+ (-1)^{k_1}\varphi\cutprod\partial^\ast\psi.
\]
Consequently, there is an induced product on magnitude cohomology
\[
	\MH_{\ell_1}^{k_1}(X)\otimes
	\MH_{\ell_2}^{k_2}(X)
	\longrightarrow
	\MH_{\ell_1+\ell_2}^{k_1+k_2}(X),
	\quad
	\alpha\otimes\beta\longmapsto \alpha\cutprod\beta
\]
that makes $\MH^\ast_\ast(X)$ into a unital, associative
bigraded ring with unit $1=[u]\in \MH^0_0(X)$.
We refer to $\MH^\ast_\ast(X)$, equipped with the
product, as the \emph{magnitude cohomology ring}.
The assignment $X\mapsto\MH^\ast_\ast(X)$,
$f\mapsto f^\ast$ is a contravariant functor from the category
of metric spaces and non-increasing maps into the category
of unital associative bigraded rings.
\end{definition}

Our definition of the product on magnitude cohomology
is similar to the definition of the cup-product
in singular cohomology. Compare
Definition~\ref{definition-magcoh} with 
Section~3.2 of~\cite{Hatcher}, say.  
But in fact we know of no direct relationship between the two,
and moreover we will see an important difference in the next proposition.

Recall from~\cite{LeinsterShulman}
that elements $x,y$ in a generalised metric space
are \emph{adjacent} if $d(x,y)$ is nonzero and finite and
$d(x,y) = d(x,a) + d(a,y)\implies a=x\text{ or }a=y$.
Observe that any graph with at least one edge,
and any finite metric space with at least two points,
contains at least one adjacent pair.

\begin{proposition}
\label{proposition-noncommutative}
Suppose that $X$ is an extended quasi-metric space 
containing an adjacent pair $(x,y)$.
Then $\MH^\ast_\ast(X)$ is not graded-commutative.
\end{proposition}

\begin{proof}
The proof of \cite[Theorem 4.3]{LeinsterShulman} 
(compare with \cite[Proposition~2.9]{HepworthWillerton})
can be dualised to show that if $(x,y)$ is an adjacent pair in $X$,
then there is a cocycle $\varphi_{xy}\in\MC^1_{d(x,y)}(X)$
defined by
\[
	\varphi_{xy}(x',y') = 
	\left\{\begin{array}{ll}
		1 & \text{if }(x',y')=(x,y),
		\\
		0 & \text{otherwise},	
	\end{array}\right.
\]
and that the cohomology classes $\uu_{xy}=[\varphi_{xy}]$
for $(x,y)$ adjacent form a basis of $\MH^1_\ast(X)$.
Note that the proof of Theorem~4.3 of~\cite{LeinsterShulman}
is stated only for metric spaces, but that the proof extends
to the extended quasi-metric case without change.
Given such an adjacent pair $(x,y)$, one may check that
$(x,y,x)\in\MC_{2,d(x,y)+d(y,x)}(X)$ is a cycle, and that
\[
	\langle \uu_{xy}\cutprod \uu_{yx},[(x,y,x)]\rangle = 1,
	\qquad
	\langle \uu_{yx}\cutprod \uu_{xy},[(x,y,x)]\rangle = 0,
\]
so that $\uu_{xy}\cutprod \uu_{yx}$ and $\uu_{yx}\cutprod \uu_{xy}$ are not equal
up to any choice of sign.
Here $\langle-,-\rangle$ denotes the Kronecker pairing
between homology and cohomology; see Remark~\ref{remark-universal}
below.
\end{proof}

Just as magnitude homology is a categorification of the magnitude
of graphs and metric spaces (see \cite[Theorem 2.8]{HepworthWillerton} and
\cite[Theorem~3.5]{LeinsterShulman}), 
the same is true of magnitude cohomology.
To see this, we use the version of the magnitude 
from~\cite{LeinsterShulman} that takes values
in the field $\hahn$ of \emph{Novikov series}.
For details on this
we refer the reader to \cite{LeinsterShulman}, in particular
Definition~3.1, Theorem~3.2, and the discussion that precedes them.

\begin{theorem}
\label{theorem-categorification}
Let $X$ be a finite quasi-metric space.
Then
\[
	\# X = \sum_{\ell\geqslant 0}\sum_{k=0}^\infty
	(-1)^k \cdot \rank(\MH_{\ell}^k(X)) \cdot q^\ell,
\]
where each sum over $k$ is finite, and the infinite sum over $\ell$
converges in the topology of $\hahn$.
\end{theorem}

\begin{proof}
Theorem~3.5 of~\cite{LeinsterShulman} gives the same result but
with 
$\rank(\MH_{k,\ell}(X))$ in place of $\rank(\MH^k_\ell(X))$.
That the two ranks are equal follows from 
the universal coefficient sequence of Remark~\ref{remark-universal}
below,
together with the fact that 
$\mathrm{Ext}(\MH_{k,\ell}(X),\Z)$ is finite
since $\MH_{k,\ell}(X)$ is finitely generated.
\end{proof}

We end this section with some remarks.

\begin{remark}[The universal coefficient sequence]
\label{remark-universal}
Magnitude homology and cohomology are related by a universal coefficient
sequence:
\begin{equation}
	0\to \mathrm{Ext}(\MH_{k-1,\ell}(X),\Z)
	\longrightarrow
	\MH^k_\ell(X)
	\longrightarrow
	\Hom(\MH_{k,\ell}(X),\Z)
	\longrightarrow
	0
\end{equation}
It is natural in $X$, and split, but not naturally split.
See Theorem~3.2 of~\cite{Hatcher}.
The second nontrivial arrow in this sequence determines 
a \emph{Kronecker pairing} that we denote
$
	\langle-,-\rangle
	\colon
	\MH^k_\ell(X)\otimes \MH_{k,\ell}(X)
	\to
	\Z
$.
\end{remark}

\begin{remark}[The two gradings]
Magnitude homology $\MH_{\ast,\ast}(X)$
and cohomology $\MH_\ast^\ast(X)$ each have two gradings,
which we usually specify as 
$\MH_{k,\ell}(X)$ and $\MH_\ell^k(X)$.
The first grading $k\in\N$ 
is the \emph{homological} or \emph{cohomological} grading,
and it comes from the grading on the underlying chain
and cochain complexes.
The second grading $\ell\in[0,\infty)$
is the \emph{length} or \emph{distance} grading,
and arises because the length of simplices is not changed
by the differential.
\end{remark}

\begin{remark}[Coefficients]
We could have defined magnitude homology and cohomology
{with coefficients} in an abelian group $A$ by the rules
\[
	\MH_{k,\ell}(X;A)=H_k(\MC_{\ast,\ell}(X)\otimes A),
	\qquad
	\MH_\ell^k(X;A)=H^k(\Hom(\MC_{\ast,\ell}(X), A))
\]
to much the same effect as the use of coefficients in singular
homology.  For the sake of simplicity we have chosen not to do so.
\end{remark}

\begin{remark}[Involutions]
The magnitude homology and cohomology 
of an extended pseudo-metric space can be equipped
with an involution given on simplices by
$(x_0,\ldots,x_k)\mapsto (-1)^\frac{k(k+1)}{2} (x_k,\ldots,x_0)$.
This makes the magnitude cohomology into a bigraded unital associative
ring with involution.
We have chosen not to investigate this structure here.

A key assumption for the involution to be defined is that
the metric space must be symmetric.
Now, a generalised metric space is symmetric if and only if it 
has the structure of a {dagger}-$[0,\infty]$-category.
(A \emph{dagger}-$\V$-category
is a $\V$-category $\mathcal{C}$ equipped with an involutive $\V$-functor
$\dagger\colon\mathcal{C}^\mathrm{op}\to \mathcal{C}$
that is the identity on objects.)
So dagger-$\V$-categories may be an appropriate setting
to which to extend the involution defined above.
\end{remark}

\section{The recovery theorem}
\label{section-recovery}

In this section we prove a `recovery' theorem showing 
that, for a large class of extended quasi-metric spaces,
the magnitude cohomology ring of the space determines the space itself
up to isometry.
In particular, if two such spaces have isomorphic magnitude cohomology
rings then the spaces themselves are isometric.
This holds in particular for finite metric spaces and for directed graphs.

\begin{theorem}[Recovery theorem]
\label{theorem-recovery}
Let $X$ be an extended quasi-metric space
for which $\inf\{d(a,b)\mid a,b\in X,\ a\neq b\}$
is positive.
Then $X$ is determined up to isometry by the
magnitude cohomology ring $\MH_\ast^\ast(X)$.
In particular, if $X$ and $X'$ are two such spaces,
and $\MH_\ast^\ast(X)\cong\MH_\ast^\ast(X')$
as bigraded rings, then $X$ and $X'$ are isometric.
\end{theorem}

The precise method by which $X$ is recovered from $\MH_\ast^\ast(X)$
will be spelled out in Remark~\ref{remark-recovery} below.
In a finite extended
quasi-metric space the nonzero distances have a nonzero minimum,
and so we obtain:

\begin{corollary}
\label{corollary-finite-metric-recovery}
If $X$ is a finite quasi-metric space, then $X$ is determined 
up to isometry by its magnitude cohomology ring $\MH^\ast_\ast(X)$.
\end{corollary}

A directed graph determines an 
extended quasi-metric space in which all distances
are at least $1$ \emph{via} the shortest path metric, 
and this metric in turn determines the graph up to isomorphism.
Thus Theorem~\ref{theorem-recovery} gives us:

\begin{corollary}
A directed graph $G$ is determined up to isomorphism
by its magnitude cohomology ring $\MH^\ast_\ast(G)$.
\end{corollary}

The results just presented are 
in extreme contrast with the situation for magnitude homology,
or for magnitude cohomology without the ring structure.
For example, any two trees with the same
number of vertices have isomorphic magnitude homology and cohomology 
groups~\cite[Corollary~6.8]{HepworthWillerton}.

The following example shows that it is impossible to
extend Corollary~\ref{corollary-finite-metric-recovery}
to arbitrary metric spaces.  

\begin{example}
\label{example-compact}
Kaneta and Yoshinaga~\cite[Corollary~5.3]{KanetaYoshinaga}
and Jubin~\cite[Corollary~7.3]{Jubin} have independently shown
that if $X$ is a convex subset of Euclidean space,
then $\MH_{k,\ell}(X)=0$ except when $k=\ell=0$.
The same conclusion therefore holds for magnitude cohomology,
so that $\MH^\ast_\ast(X)$ is zero in all bidegrees except for
$\MH^0_0(X)\cong\Z^X$.  
Thus the magnitude cohomology ring of convex subsets 
of Euclidean space determines only the cardinality of the underlying
set $X$, and cannot recover the metric on $X$.
\end{example}

We now move on to the proof of Theorem~\ref{theorem-recovery}.
To obtain the recovery result it is in fact enough to look in
homological degrees $k=0,1$, and so we begin 
by determining $\MH_\ast^\ast(X)$
in these degrees. This is based on the homological results
in~\cite[Section~4]{LeinsterShulman}.
Recall that a pair $x,y\in X$ is 
\emph{adjacent} if $d(x,y)$ is nonzero and finite, and
$d(x,y) = d(x,a) + d(a,y)\implies a=x\text{ or }a=y$.
We write $\adj(X,\ell)$ for the set of \emph{ordered}
pairs $(x,y)$ in $X$ such that $x,y$ are adjacent and $d(x,y)=\ell$.
Given a set $A$, we will write $\Z^A$ for the abelian
group of all functions $A\to\Z$ under pointwise addition.

\begin{proposition}\label{proposition-zeroone}
Let $X$ be an extended quasi-metric space.
Then $\MH^0_\ell(X)=0$ if $\ell>0$, and there are natural isomorphisms:
\begin{align}
	\MH^0_0(X)&\cong \Z^X 
	\label{isomorphism-zero}
	\\
	\MH^1_\ell(X)&\cong\Z^{\adj(X,\ell)}
	\label{isomorphism-one}
\end{align}
The isomorphism~\eqref{isomorphism-zero} 
is an isomorphism of rings, where $\Z^X$ is equipped
with pointwise multiplication.
And the isomorphism~\eqref{isomorphism-one}
identifies the $\MH_0^0(X)$-bimodule $\MH^1_\ell(X)$
with the $\Z^X$-bimodule $\Z^{\adj(X,\ell)}$ determined
by the rule
\[
	(f\cdot m \cdot g)(x,y)
	=
	f(x)m(x,y)g(y)
\] 
for $f,g\in\Z^X$ and $m\in\Z^{\adj(X,\ell)}$.
\end{proposition}

\begin{proof}
The isomorphisms~\eqref{isomorphism-zero} and~\eqref{isomorphism-one}
are obtained by dualising the proofs of Theorems~4.1 and~4.3
of~\cite{LeinsterShulman}.
(Those results were only stated in the metric case,
but extend to the extended quasi-metric case without change.)
In addition this shows that $f\in\Z^X$ corresponds to the 
cohomology class of the element $\varphi_f\in\MC^0_0(X)$
defined by $\varphi_f(x_0) = f(x_0)$ for each $0$-simplex $(x_0)$,
and that $m\in\Z^{\adj(X,\ell)}$ corresponds to the cohomology
class of the element $\psi_m\in\MC^1_\ell(X)$
defined by
\[
	\psi_m(x,y)=
	\left\{\begin{array}{ll}
		m(x,y) & \text{if }(x,y)\text{ adjacent,}
		\\
		0 & \text{otherwise.}	
	\end{array}\right.
\]
We compute
\[
	(\varphi_f\cutprod\varphi_g)(x)
	=\varphi_f(x)\cdot\varphi_g(x)
	=f(x)g(x)
\]
so that $\varphi_f\cutprod\varphi_g=\varphi_{fg}$.
And we compute
\[
	(\varphi_f\cutprod\psi_m\cutprod\varphi_g)(x,y)
	=
	\varphi_f(x)
	\cdot
	\psi_m(x,y)
	\cdot
	\varphi_g(y)
	=
	f(x)m(x,y)g(x)
\]
so that $\varphi_f\cutprod\psi_m\cutprod\varphi_g = \psi_{f\cdot m\cdot g}$.
The induced relations on cohomology classes prove that our
isomorphisms respect the multiplicative structures
as described.
\end{proof}

\begin{proposition}
\label{proposition-recovery}
The magnitude cohomology ring $\MH_\ast^\ast(X)$ of an
extended quasi-metric space
$X$ determines the underlying set $X$ up to bijection, together with 
the adjacent pairs in $X$ and the distances between them.
\end{proposition}

\begin{proof}
Given $x\in X$, let $\delta_x\in\Z^X$ denote the function
with value $1$ on $x$ and $0$ on all other elements of $X$.
Then the primitive idempotents of $\Z^X$ are precisely the 
elements $\delta_x$.  Given $x,y\in X$ and $m\in\Z^{\adj(X,\ell)}$,
we have
\[
	(\delta_x \cdot m\cdot \delta_y)(a,b)
	=\left\{\begin{array}{ll}
		m(x,y)& \text{if }(a,b)=(x,y),
		\\
		0 & \text{if }(a,b)\neq(x,y).
	\end{array}\right.
\]
Note that the first possibility only occurs if $(x,y)\in\adj(X,\ell)$.
Thus $\delta_x\cdot\Z^{\adj(X,\ell)}\cdot \delta_y$ is nonzero if and only if 
$(x,y)\in\adj(X,\ell)$.

Thus, using the isomorphisms of Proposition~\ref{proposition-zeroone},
we see that the magnitude cohomology ring of $X$ determines
$X$ up to bijection as the set of primitive idempotents of $\MH^0_0(X)$.
And given primitive idempotents $e,f\in\MH^0_0(X)$, we have
$e\cdot\MH^1_\ell(X)\cdot f\neq 0$ 
if and only if $e,f$ correspond to elements that
are adjacent and a distance $\ell$ apart.
\end{proof}

\begin{lemma}
\label{lemma-recovery}
Let $X$ be an extended quasi-metric space
for which $\inf\{d(a,b)\mid a,b\in X,\ a\neq b\}$
is positive.
Then for any distinct $a,b\in X$, $d(a,b)$
is the minimum of the set
\[
	\{d(x_0,x_1)+\cdots+d(x_{k-1},x_k)\mid
	a=x_0,\ b=x_k,\ x_{i-1},x_i\text{ adjacent for }i=1,\ldots,k\}
\]
if the set is nonempty, and $d(a,b)=\infty$ otherwise.
\end{lemma}

\begin{proof}
Let $a,b$ be distinct elements of $X$,
and let us write $A_{a,b}$ for the set given in the statement.
If $d(a,b)=\infty$ then $A_{a,b}$ is empty and the result follows.
If $d(a,b)$ is finite, then any element of $A_{a,b}$ is greater than
or equal to $d(a,b)$ by the triangle inequality.  So it remains
to show that $A_{a,b}$ is nonempty and contains $d(a,b)$.
To do this, we iteratively construct sequences $x_0,x_1,\ldots,x_k$
with the following properties.
\begin{itemize}
	\item
	$a=x_0$, $b=x_k$.
	\item
	Consecutive entries are distinct.
	\item
	$d(x_0,x_1)+\cdots+d(x_{k-1},x_k)=d(a,b)$.
\end{itemize}
We do this by starting with the sequence $x_0=a$, $x_1=b$.
Given such a sequence, if its consecutive entries are not all adjacent, 
then there is some pair $x_{i-1},x_i$ which is not adjacent.  
We may then insert a new entry between $x_{i-1}$ and $x_i$ 
to obtain a longer sequence with the same properties.
The length $k$ of any such sequence is bounded above by 
$d(a,b)/\inf\{d(x,y)\mid x,y\in X,\ x\neq y\}$, and so this process must end with a sequence 
in which consecutive pairs are adjacent.
This sequence demonstrates that $A_{x,y}$ is nonempty and contains $d(a,b)$,
and this completes the proof.
\end{proof} 

\begin{proof}[Proof of Theorem~\ref{theorem-recovery}]
Lemma~\ref{lemma-recovery} shows that $X$ is determined by 
the data of the underlying set, the adjacent pairs,
and the distances between them.  And Proposition~\ref{proposition-recovery}
shows that this data is determined by the magnitude
cohomology ring $\MH^\ast_\ast(X)$.
\end{proof}

\begin{remark}\label{remark-recovery}
We can now specify how to recover $X$ from $\MH^\ast_\ast(X)$
as in Theorem~\ref{theorem-recovery}.
Let $\bar X$ denote the set of primitive idempotents of $\MH^\ast_\ast(X)$.  For $e\in\bar X$
we declare $\bar d(e,e)=0$.  Next, for $e,f\in\bar X$, compute $e\cdot\MH^1_\ell(X)\cdot f$ for each $\ell$.  There is at most one value of $\ell$ for which the resulting group is nonzero.  If there is indeed such an $\ell$, then we declare $e$ and $f$ to be adjacent and set $\bar d(e,f)=\ell$.  Next, for each pair $e$, $f$ that is not adjacent, we define $\bar d(e,f)$ to be the minimum of the set
\[
        \{\bar d(x_0,x_1)+\cdots+\bar d(x_{k-1},x_k)\mid
        a=x_0,\ b=x_k,\ x_{i-1},x_i\text{ adjacent for }i=1,\ldots,k\}
\]
if the set is nonempty, and we set $\bar d(a,b)=\infty$ otherwise.  Then $\bar d$ is an extended quasi-metric on $\bar X$, and $(X,d)$ is isometric to $(\bar X,\bar d)$.
\end{remark}

\section{Magnitude homology of enriched categories}
\label{section-enriched-homology}

In this brief section
we recall Leinster and Shulman's definition of magnitude homology 
of enriched categories~\cite{LeinsterShulman}, 
and we spell out the details in the case of posets,
categories, and generalised metric spaces.
This is intended to motivate and facilitate the introduction of magnitude
cohomology in the following section,
but we also hope that it will give readers who are not
familiar with~\cite{LeinsterShulman}
a quick way into the subject.
Of course, we heartily recommend the original treatment,
namely section~5 of~\cite{LeinsterShulman}.

Let $\V$ be a symmetric monoidal category,
let $\A$ be a closed symmetric monoidal abelian category,
and let $\Sigma\colon\V\to\A$ be a strong symmetric monoidal functor.
We assume that $\V$ is \emph{semicartesian}, meaning that 
the unit object $1_\V$ is terminal.

\begin{definition}[Magnitude homology]
Given a $\V$-category $X$, the \emph{magnitude nerve} of $X$
is the simplicial object $B^\Sigma_\bullet(X)$ in $\A$ defined by
\[
	B^\Sigma_k(X)
	=\bigoplus_{x_0,\ldots,x_k}
	\Sigma X(x_0,x_1)\otimes\cdots\otimes\Sigma X(x_{k-1},x_k)
\]
where the sum is over all tuples $x_0,\ldots,x_k$ of objects of $X$.
The inner face maps $d_1,\ldots,d_{k-1}$ are defined using monoidality 
of $\Sigma$ and composition in $X$, and have the effect of replacing two
adjacent factors $\Sigma X(x_{i-1},x_i)\otimes\Sigma X(x_i,x_{i+1})$
with a single factor $\Sigma X(x_{i-1},x_{i+1})$.
The outer face maps $d_0$ and $d_k$ are defined using terminality
of $1_\V$ and monoidality of $\Sigma$, and have the effect of
erasing the first and last factors $\Sigma X(x_0,x_1)$
and $\Sigma X(x_{k-1},x_k)$ respectively.
The degeneracy maps $s_i$ are defined using the identity maps
of $X$ and monoidality of $\Sigma$, and have the effect of inserting
a factor $\Sigma X(x_i,x_i)$ between
$\Sigma X(x_{i-1},x_i)$ and $\Sigma X(x_i,x_{i+1})$.
We leave it to the reader to write out the simplicial structure maps
in detail for themselves, or to unpack them
from Remark~5.11 of~\cite{LeinsterShulman} if they wish.

The \emph{magnitude chain complex} $C^\Sigma_\ast(X)$ of $X$
is defined to be the chain complex $C_\ast(B^\Sigma_\bullet(X))$
of $B^\Sigma_\bullet(X)$.
Thus
\[
	C^\Sigma_k(X)
	=
	\bigoplus_{x_0,\ldots,x_k}
	\Sigma X(x_0,x_1)\otimes\cdots\otimes\Sigma X(x_{k-1},x_k)
\]
and
\[
	\partial\colon C_k^\Sigma(X)
	\longrightarrow C_{k-1}^\Sigma(X)
\]
is defined by $\partial = d_0-d_1+\cdots+(-1)^kd_k$.
The \emph{normalized magnitude chain complex} $N_\ast^k(X)$ of $X$
is the normalised chain complex $N_\ast(B_\bullet^\Sigma(X))$
of $B_\bullet^\Sigma(X)$.
This is the quotient of $C_\ast(B_\bullet^\Sigma(X))$
by the subcomplex generated by the images of the degeneracy maps.
We refer the reader to sections~8.2 and 8.3 of~\cite{Weibel}.
The \emph{magnitude homology} $H^\Sigma_\ast(X)$ of $X$ is the homology
of the magnitude chains of $X$, or equivalently the homology
of the normalised magnitude chains of $X$:
\[
	H^\Sigma_k(X)
	=
	H_k(C^\Sigma_\ast(X))
	\cong
	H_k(N_\ast^\Sigma(X)).
\]
\end{definition}

\begin{remark}\label{remark-assumption}
	In \cite{LeinsterShulman} the complex
	$C^\Sigma_\ast(X)$ is denoted by 
	$\widetilde{\mathrm{MC}}_\ast^\Sigma(X)$,
	while $N^\Sigma_\ast(X)$ is denoted by $\mathrm{MC}_\ast^\Sigma(X)$.
\end{remark}

\begin{example}[Posets]
\label{example-posets}
	Let $\V$ be the category $\two$, 
	with objects $\true$ and $\false$ for `true'
	and `false' respectively, with a single morphism
	$\false\to\true$ besides the identities, and
	with $\otimes$ given by conjunction, i.e.~logical `and'. Thus
	$1_\two=\true$.
	We define $\Sigma\colon\two\to\Ab$ by 
	$\Sigma(\true)=\Z$, $\Sigma(\false)=0$, 
	with monoidal structure in which
	the maps 
	$1_\Ab\to\Sigma(1_\two)$ and $\Sigma(\true)\otimes\Sigma(\true)
	\to\Sigma(\true)$ are the identity and multiplication
	maps $\Z\to\Z$ and $\Z\otimes\Z\to\Z$ respectively.
	A skeletal category $X$ enriched in $\two$ is nothing other
	than a poset: the objects of $X$ are the elements,
	and $x\leqslant y$ if and only if $X(x,y)=\true$.
	Let $X$ be a poset, regarded as a skeletal $\two$-category.
	We will describe $B_\bullet^\Sigma(X)$.
	Observe that there is an isomorphism	
	\[
		\Sigma X(x_0,x_1)\otimes\cdots\otimes\Sigma X(x_{k-1},x_k)
		\cong
		\left\{\begin{array}{ll}
			\Z & \text{if }x_0\leqslant\cdots\leqslant x_k
			\\
			0 & \text{otherwise}
		\end{array}\right.
	\]
	given by the product $\Z^{\otimes k}\to\Z$.
	Thus $B^\Sigma_k(X)\cong\Z\{x_0\leqslant\cdots\leqslant x_k\}$,
	and one can check that the face and degeneracy
	maps are given by erasing elements and inserting equalities,
	respectively.
	Thus $N^\Sigma_k(X)\cong\Z\{x_0<\cdots<x_k\}$, with boundary
	map given by the usual alternating sum of faces.
	In other words $N^\Sigma_\ast(X)$ is precisely the simplicial chain
	complex of the order complex $|X|$ of $X$, and
	\[
		H^\Sigma_\ast(X) = H_\ast(|X|).
	\]
\end{example}

\begin{example}[Categories]
\label{example-categories}
	Let $\V=\mathrm{Set}$ be the category of sets.
	A category enriched in $\mathrm{Set}$ is nothing other
	than a category.
	Let $\A=\Ab$, and let $\Sigma\colon\mathrm{Set}\to\Ab$
	be the free abelian group functor with its evident monoidal structure.
	Thus if $X$ is a category and $x,y$ are objects, then
	$\Sigma X(x,y)$ is the free abelian group on the morphisms
	$f\colon x\to y$.
	Consequently, there is an isomorphism
	\begin{multline*}
		B_k^\Sigma(X)=
		\bigoplus_{x_0,\ldots,x_k}
		\Sigma X(x_0,x_1)\otimes\cdots\otimes\Sigma X(x_{k-1},x_k)
		\\
		\xrightarrow{\ \ \cong\ \ }
		\Z\left\{ x_0\xrightarrow{f_1}x_1\xrightarrow{f_2}\cdots
		\xrightarrow{f_k} x_k
		\right\}
		=
		\Z N_k(X)
	\end{multline*}
	where $N_\bullet(X)$ denotes the simplicial nerve of $X$.
	Unwinding the definition of the face and degeneracy maps
	shows that $B^\Sigma_\bullet(X)$ is precisely
	the free abelian group on $N_\bullet X$.
	Thus $C^\Sigma_\ast(X)$ is the simplicial
	chains on $N_\bullet(X)$, or equivalently,
	the simplicial chains on the classifying space
	$BX$, and so we have
	\[
		H^\Sigma_\ast(X)
		=
		H_\ast(BX).
	\]
\end{example}

\begin{example}
\label{example-metric}
	Now let us take $\V=[0,\infty]$, 
	so that a category enriched in $\V$ is a generalised
	metric space.  
	(See the introduction to section~\ref{section-magcoh-metric}.)
	Let us take $\A=\prod_\R\Ab$, the category of $\R$-graded
	abelian groups, equipped with the symmetric monoidal structure 
	given by $(A\otimes B)_\ell = \bigoplus_{j+k=\ell}A_j\otimes B_k$.
	Now we define $\Sigma\colon[0,\infty]\to\A$
	to be the functor which sends $\ell\in[0,\infty]$ to 
	a copy of $\Z$
	concentrated in degree $\ell$, and which necessarily sends
	all non-identity morphisms in $[0,\infty]$ to the zero map.
	We equip $\Sigma$ with the symmetric monoidal structure
	under which $\Sigma(j)\otimes\Sigma(k)\to\Sigma(j+k)$
	is given in degree $j+k$ by the multiplication map $\Z\otimes\Z\to\Z$,
	and under which $1_\A\to\Sigma(0)$ is given in degree $0$
	by the identity map $\Z\to\Z$.

	Now if $X$ is a generalised metric space,
	then $\Sigma X(x_0,x_1)\otimes\cdots\otimes\Sigma X(x_{k-1},x_k)$
	is the tensor product of $k$ copies of $\Z$, concentrated
	in degrees $d(x_0,x_1),\ldots,d(x_{k-1},x_k)$ respectively.
	This is canonically isomorphic, under the multiplication map,
	to a single copy of $\Z$ concentrated in degree $\ell(x_0,\ldots,x_k)$.
	If we write the generator of this copy of $\Z$ as
	$(x_0,\ldots,x_k)$, then we find that
	\[
		B^\Sigma_k(X) = \Z\{(x_0,\ldots,x_k)\mid
		x_0,\ldots,x_k\in X\},
	\]
	where the right-hand-side is interpreted as an $\R$-graded abelian
	group in the evident way.
	The reader may now be able to verify that the face map
	$d_i$ is given by
	\[
		d_i(x_0,\ldots,x_k)
		=
		\left\{\begin{array}{ll}
			(x_0,\ldots\widehat{x_i},\ldots,x_k)
			&
			\text{if }\ell(x_0,\ldots\widehat{x_i},\ldots,x_k)
			=  \ell(x_0,\ldots,x_k),
			\\
			0
			&
			\text{otherwise}
		\end{array}\right.
	\]
	and that the degeneracy map $s_i$ is given by
	\[
		s_i(x_0,\ldots,x_k)=(x_0,\ldots,x_i,x_i,\ldots,x_k).
	\]
	Thus the image of $s_i$ consists of tuples whose $i$-th
	entry is repeated, and so the span of the images of these degeneracy
	maps is exactly the span of the tuples which have at least one repeated
	consecutive entry.  Dividing out by this span,
	we see that the normalised magnitude chains 
	$N^\Sigma_\ast(X)$ are precisely the magnitude chains of $X$
	as defined in section~\ref{section-magcoh-metric},
	\[
		N^\Sigma_\ast(X) = \MC_{\ast,\ast}(X)
	\]
	and consequently
	\[
		H^\Sigma_\ast(X) = \MH_{\ast,\ast}(X).
	\]
\end{example}

\section{Magnitude cohomology of enriched categories}
\label{section-enriched-cohomology}

In this section we will define the magnitude cohomology ring
of an enriched category, and we will give examples showing
that this recovers the cohomology ring of the order complex
of a poset, the cohomology ring of the classifying space of a category,
and the magnitude cohomology ring of a metric space. 
Finally we prove that when the enriching category is cartesian,
the magnitude cohomology ring is graded-commutative.
This explains the commutativity of the first two examples
above, since the underlying categories $\V=\two$ (for posets)
and $\V=\set$ (for categories) are both cartesian,
while $\V=[0,\infty]$ (for generalised metric spaces) is not.

Throughout this section we fix a semicartesian symmetric monoidal category $\V$,
a closed symmetric monoidal abelian category $\A$,
and a strong symmetric monoidal functor $\Sigma\colon\V\to\A$.
We write $1_\A$ for the unit object of $\A$, and we write
$[-,-]$ for hom-objects in $\A$.

\begin{definition}[Magnitude cohomology of enriched categories]
Let $X$ be a $\V$-category.
The \emph{magnitude cochain complex} $C_\Sigma^\ast(X)$ of $X$
is the cochain complex in $\A$ obtained by setting
\[
	C_\Sigma^k(X) = [C_k^\Sigma(X),1_\A],
\]	
with the induced differential $\partial^\ast=[\partial,1_\A]$.
The \emph{normalized magnitude cochain complex} 
$N_\Sigma^\ast(X)$ of $X$ is defined by
\[
	N_\Sigma^k(X) = [N_k^\Sigma(X),1_\A]
\]
with $\partial^\ast=[\partial,1_\A]$.  
The \emph{magnitude cohomology} $H_\Sigma^\ast(X)$ is defined
to be the cohomology of the magnitude cochains, or equivalently
of the normalized magnitude cochains:
\[
	H_\Sigma^\ast(X) = H^\ast(C_\Sigma^\ast(X))\cong
	H^\ast(N_\Sigma^\ast(X)).
\]
\end{definition}

\begin{definition}[The coproduct and counit]
\label{definition-coproduct}
Let $X$ be a $\V$-category.
We define the \emph{coproduct}
\[
	\Delta\colon C^\Sigma_\ast(X)
	\longrightarrow
	C^\Sigma_\ast(X)\otimes C^\Sigma_\ast(X)
\]	
to be the sum of the maps 
\[
	\Delta\colon C^\Sigma_{p+q}(X)
	\to
	C^\Sigma_p(X)\otimes C^\Sigma_q(X)
\]
that send the $x_0,\ldots,x_{p+q}$ summand into the product
of the $x_0,\ldots,x_p$ and $x_p,\ldots,x_{p+q}$ summands by the evident
map from 
\[
	\Sigma X(x_0,x_1)\otimes\cdots\otimes\Sigma X(x_{p+q-1},x_{p+q})
\]
to
\[
	\left(
		\Sigma X(x_0,x_1)\otimes\cdots\otimes\Sigma X(x_{p-1},x_{p})
	\right)
	\otimes
	\left(
		\Sigma X(x_p,x_{p+1})\otimes
		\cdots\otimes\Sigma X(x_{p+q-1},x_{p+q})
	\right).
\]
It is a chain map.  We further define the \emph{counit} map
\[
	\varepsilon\colon C^\Sigma_\ast(X)
	\longrightarrow	
	1_{\Ch_\A}	
\]
to be the map that is given by $0$ in positive degrees
and in degree $0$ by the map $C_0^\Sigma(X)\to 1_{\Ch\A}$,
$\bigoplus_{x_0}1_\A\to 1_\A$
that is the identity on each summand.
Again, this is a chain map. 
Furthermore, the coproduct and counit both reduce to maps
on the normalized magnitude chains,
that we also call the \emph{coproduct} and \emph{counit}
\[
	\Delta\colon N^\Sigma_\ast(X)
	\longrightarrow
	N^\Sigma_\ast(X)\otimes N^\Sigma_\ast(X),
	\qquad
	\varepsilon\colon N^\Sigma_\ast(X)\to 1_{\Ch_\A}.
\]	
These maps make $C_\ast^\Sigma(X)$ and $N_\ast^\Sigma(X)$
into coassociative, counital differential graded coalgebras
in $\A$.
\end{definition}

\begin{definition}[The magnitude cohomology ring]
The maps $\Delta$ and $\varepsilon$ induce dual maps
\[
	\mu\colon C^\ast_\Sigma(X)\otimes C^\ast_\Sigma(X)
	\longrightarrow C^\ast_\Sigma(X),
	\qquad
	\eta\colon 1_{\Ch_\A} \longrightarrow C^\ast_\Sigma(X)
\]
defined by $\mu=[\Delta,1_\A]$ and $\eta=[\varepsilon,1_\A]$
making $C^\ast_\Sigma(X)$ into an associative, unital dg-algebra
in the abelian category $\A$.
They induce maps of the same name in homology 
\[
	\mu\colon H^\ast_\Sigma(X)\otimes H^\ast_\Sigma(X)
	\longrightarrow H^\ast_\Sigma(X),
	\qquad
	\eta\colon 1_{\A} \longrightarrow H^\ast_\Sigma(X)
\]
and these make $H^\ast_\Sigma(X)$ into an associative unital graded algebra
in $\A$.
The same definitions can be carried out using the normalized chains,
and produce the same structure on $H^\ast_\Sigma(X)$.
\end{definition}

\begin{example}[Magnitude cohomology rings of posets]
Following on from Example~\ref{example-posets}, and using the isomorphism
$N^\Sigma_k(X)\cong\Z\{x_0< \cdots < x_k\}$
established there, one finds that $N^\ast_\Sigma(X)$ is the $\Z$-dual
to $N_\ast^\Sigma(X)$, i.e.~the usual simplicial cochains on the order
complex $|X|$, together with the product defined by the
formula 
$(\xi\cdot\eta)(x_0<\cdots<x_k)=\xi(x_0<\cdots<x_i)\eta(x_i<\cdots<x_k)$
for $\xi\in N^i_\Sigma(X)$ and $\eta\in N^{k-i}_\Sigma(X)$.
This is again the standard definition of the cochain-level cup-product
on $|X|$.  Thus we have an isomorphism of graded associative algebras
\[
	H^\ast_\Sigma(X)\cong H^\ast(|X|).
\]
Note that the algebra on the right-hand-side is graded commutative.
\end{example}

\begin{example}[Magnitude cohomology rings of categories]
Following on from Example~\ref{example-categories}, and using the 
identification of $B_\bullet^\Sigma(X)$ with the free abelian
group on the simplicial nerve,
$\Z N_\bullet (X)$, we find that $C^\ast_\Sigma(X)$ is the simplicial
cochain complex of $N_\bullet(X)$, or equivalently the simplicial
cochain complex of the classifying space $BX$.
Moreover, unwinding the definition of the product shows that it
again coincides with the usual definition of the cochain-level
cup product, so that
\[
	H^\ast_\Sigma(X)\cong H^\ast(BX)
\]
as associative graded unital rings.
Note again that the right hand side is graded commutative.
\end{example}

\begin{example}[Magnitude cohomology rings of metric spaces]
Following on from Example~\ref{example-metric},
and using the identification $N_\ast^\Sigma(X)\cong\MC_{\ast,\ast}(X)$
obtained there, one immediately obtains $N^\ast_\Sigma(X)\cong
\MC^\ast_\ast(X)$. 
Recall that the isomorphism $N_k^\Sigma(X)\cong\MC_{k,\ast}(X)$
identifies the generator $1\otimes\cdots\otimes 1\in
\Sigma X(x_0,x_1)\otimes\cdots\otimes \Sigma X(x_{k-1},x_k)$
with the simplex $(x_0,\ldots,x_k)$.
Thus the map $\Delta$ of Definition~\ref{definition-coproduct}, 
after translating it to a coproduct
on $N^\Sigma_\ast(X)$, is the map that sends $(x_0,\ldots,x_k)$
to $\sum_{i=0}^k (x_0,\ldots,x_i)\otimes(x_i,\ldots,x_k)$.
It now follows that the induced product on
$\MC_\ast^\ast(X)$ is precisely the one defined in
Definition~\ref{definition-magcoh}.
Thus
\[
	H^\ast_\Sigma(X)\cong\MH^\ast_\ast(X)
\]
is an isomorphism of rings.
\end{example}

Observe that in the first two examples above, the magnitude
cohomology rings were graded-commutative, but that in the third they
were not.  They key difference here is that in the first two cases
the enriching categories $\V=\two$ and $\V=\set$ are cartesian,
while $\V=[0,\infty]$ is not.  
The remainder of this section is given to the proof of the following
theorem.  Along the way, we will see an explicit connection
between our product and the Alexander-Whitney map
in the cartesian case.

\begin{theorem}
\label{theorem-commutativity}
Suppose that the enriching category $\V$ is cartesian,
and let $X$ be a $\V$-category.  Then $H^\ast_\Sigma(X)$
is \emph{graded commutative}.
\end{theorem}

We now work towards the proof of this theorem.
Since $\V$ is cartesian, each object $A$ of $\V$ admits
a \emph{diagonal} $\delta_A\colon A\to A\otimes A$,
and this is natural in $A$.
Now, for each object $A$ of $\V$, we obtain
a \emph{diagonal} $\delta_{\Sigma A}\colon\Sigma A \to \Sigma A\otimes\Sigma A$,
defined as the composite
\[
	\Sigma(A)
	\xrightarrow{\Sigma(\delta_A)}
	\Sigma(A\otimes A)
	\cong
	\Sigma A \otimes \Sigma A.
\]
This map $\delta_{\Sigma A}$ is natural and additive.
Moreover, it commutes with the braiding,
in the sense that the diagram
\[\xymatrix{
	{}
	&
	\Sigma A \ar[dl]_{\delta_{\Sigma A}}\ar[dr]^{\delta_{\Sigma A}}
	&
	{}
	\\
	\Sigma A\otimes\Sigma A
	\ar[rr]_\tau
	&
	{}
	&
	\Sigma A\otimes \Sigma A
}\]
commutes, where $\tau$ denotes the braiding of $\A$.

We now construct a \emph{diagonal map}
$\delta_B\colon B^\Sigma_\bullet(X)
\to B^\Sigma_\bullet(X)\otimes B^\Sigma_\bullet(X)$.
This map sends the summand corresponding to
$x_0,\ldots,x_k$ into the product of the summands
corresponding to the same sequence:
\begin{multline*}
	\Sigma X(x_0,x_1)\otimes\cdots\otimes\Sigma X(x_{k-1},x_k)
	\longrightarrow
	\\
	\left(\Sigma X(x_0,x_1)\otimes\cdots\otimes\Sigma X(x_{k-1},x_k)\right)
	\otimes
	\left(\Sigma X(x_0,x_1)\otimes\cdots\otimes\Sigma X(x_{k-1},x_k)\right)
\end{multline*}
It does so by using the diagonal map $\delta_{\Sigma X(x_{i-1},x_i)}$
for each factor, and then reshuffling the factors.
The map $\delta_B$ is indeed simplicial, and it again commutes
with the braiding.
(The verification that $\delta_B$ is simplicial is a straightforward but lengthy diagram chase.  
The verification
can be broken into cases, one for each
tuple $x_0,\ldots,x_k$ and each face or degeneracy map originating in the
corresponding summand of $B_k^\Sigma(X)$.
And in each case, the resulting square can be shown to commute by
using the definitions of the morphisms involved,
together with basic properties of $\V$, $\A$ and $\Sigma$.)

Now recall the Alexander-Whitney map.
Given an abelian category $\A$,
and simplicial objects $U$ and $V$ in $\A$,
the Alexander-Whitney map
$\AW\colon C_\ast(U\otimes V)\to C_\ast(U)\otimes C_\ast(V)$
is defined in degree $k$ to be the sum of the maps
\[
	(d_{p+1}\circ\cdots\circ d_k)\otimes (d_0\circ\cdots\circ d_0)
	\colon
	C_k(U\otimes V)=U_k\otimes V_k
	\longrightarrow 
	U_p\otimes V_q=C_p(U)\otimes C_q(V)
\]
for $p+q=k$.
It is natural, and the square
\[\xymatrix{
	C_\ast(U\otimes V) 
	\ar[r]^-{\AW}\ar[d]_{C_\ast(\tau)}
	&
	C_\ast(U)\otimes C_\ast(V)
	\ar[d]^\tau
	\\	
	C_\ast(V\otimes U) 
	\ar[r]_-{\AW}
	&
	C_\ast(V)\otimes C_\ast(U)
}\]
commutes up to chain homotopy, where $\tau$ denotes the braiding maps. 
(We prove the last claim 
using the classical papers~\cite{EM1} and~\cite{EM2} of Eilenberg-MacLane:
The Eilenberg-Zilber
map $\nabla\colon C_\ast(U)\otimes C_\ast(V)
\to C_\ast(U\otimes V)$ defined in \cite[(5.3)]{EM1}
is a chain homotopy inverse to $\AW$ by~\cite[Theorem~2.1]{EM2},
and the Eilenberg-Zilber map $\nabla$ commutes with the braiding
~\cite[Theorem~5.2]{EM1}, so that $\AW$ commutes with the braiding
up to chain homotopy.
The proofs in Eilenberg-MacLane are only stated in the case
$\A=\Ab$.  However, all maps involved are
\emph{FD}-operators~\cite[\S 3]{EM1}, or in other words $\Z$-linear
combinations of maps induced by maps $\beta\colon [p]\to [q]$,
and all verifications take place within the group of FD-operators,
so that all definitions and verifications can be transported
directly to the setting of an arbitrary $\A$.)

\begin{lemma}
\label{lemma-composite}
The composite
\begin{align*}
	C^\Sigma_\ast(X)
	&=C_\ast(B_\bullet^\Sigma(X))
	\\
	&\xrightarrow{C_\ast\delta_B}
	C_\ast(B_\bullet^\Sigma(X)\otimes B_\bullet^\Sigma(X))
	\\
	&\xrightarrow{\AW}
	C_\ast(B_\bullet^\Sigma(X))\otimes C_\ast(B_\bullet^\Sigma(X))
	\\
	&=
	C_\ast^\Sigma(X)\otimes C_\ast^\Sigma(X).
\end{align*}
is the chain level coproduct $\Delta$.
\end{lemma}

\begin{proof}
Let us work in degree $k\geqslant 0$.
The domain of the map is the direct sum over sequences
$x_0,\ldots,x_k$ of the objects
\[
	\Sigma X(x_0,x_1)\otimes\cdots\otimes\Sigma X(x_{k-1},x_k).
\]
And the codomain of the map is the direct sum over pairs of sequences
$x_0,\ldots,x_p$, $y_p,\ldots,y_{k}$ of the objects
\[
	\left(\Sigma X(x_0,x_1)\otimes\cdots\otimes\Sigma X(x_{p-1},x_p)\right)
	\otimes
	\left(\Sigma X(y_p,y_{p+1})
	\otimes\cdots\otimes\Sigma X(y_{k-1},y_k)\right).
\]
Restricting to the $x_0,\ldots,x_k$ summand of the domain, 
$\AW\circ C_\ast\delta_B$ is the sum over $p+q=k$ of the maps
\[
	\left(
		(d_{p+1}\circ\cdots\circ d_k)\otimes (d_0\circ\cdots\circ d_0)
	\right)
	\circ
	\delta_B
\]
which land in the $x_0,\ldots,x_p$, $x_p,\ldots,x_k$ summand of
the codomain.  It therefore remains to show that this composite
is precisely the `rebracketing' map appearing in the definition
of $\Delta$.  This is a tedious but routine verification
that we leave to the reader.
\end{proof}

\begin{corollary}
If $\V$ is cartesian then
the chain level coproduct map $\Delta$ is cocommutative
up to chain homotopy, and the cochain level product
map $\mu$ is commutative up to chain homotopy.
\end{corollary}

\begin{proof}
The claim about $\Delta$ follows from Lemma~\ref{lemma-composite}
together with the fact that $\delta_B$ commutes with the braiding
and that $\AW$ commutes with the braiding up to chain homotopy.
And the claim about $\mu$ follows from that about $\Delta$.
\end{proof}

Theorem~\ref{theorem-commutativity} now follows from the corollary
by taking homology.

\section{Magnitude cohomology of finite graphs}
\label{section-graph}

We now restrict our attention to finite (undirected) graphs, which we regard
as extended metric spaces by equipping their vertex sets
with the shortest-path metric.
Since all distances in a graph are integers, the magnitude
cohomology groups $\MH^k_\ell(G)$ of a graph $G$
are concentrated in bidegrees where $k$ and $\ell$ are
both non-negative integers.

The author and Willerton in~\cite[Section 7]{HepworthWillerton}
identified an important class of graphs
called \emph{diagonal graphs}.  These are graphs 
whose magnitude homology is concentrated on the diagonal, 
i.e.~$\MH_{k,\ell}(G)=0$ whenever $k\neq \ell$.
The magnitude of a diagonal graph is the power series 
$\sum_{\ell\geqslant 0}(-1)^\ell\cdot\rank(\MH_{\ell,\ell}(G))\cdot q^\ell$,
so that its coefficients alternate in sign, 
and indeed all known cases of graphs whose magnitude
has alternating coefficients are in fact diagonal.
Moreover, 
the magnitude of a diagonal graph also 
determines the magnitude homology up to isomorphism.
We showed in~\cite{HepworthWillerton}
that diagonality is preserved under cartesian
products and projecting decompositions, and that
any join of graphs is diagonal, so that
examples of diagonal graphs include complete graphs,
discrete graphs, trees, and complete multipartite graphs.
Gu~\cite{Gu} has shown that the icosahedral
graph is diagonal, and that all
\emph{pawful} graphs (a class which includes all joins
but contains more examples) are diagonal.

In this section we give a complete description of the
\emph{diagonal part} of the magnitude cohomology ring of a finite graph $G$,
by which we mean the graded subring consisting of the groups $\MH^k_k(G)$
for $k\geqslant 0$,  and we use it to give a complete description
of the magnitude cohomology rings of any diagonal graph,
which we then make explicit in several examples.
Our results relate the magnitude cohomology of a graph
to the path algebra of the associated quiver.

\begin{definition}[Path algebra of a graph]
Let $G$ be a graph.  An \emph{edge path} in $G$
is a sequence $x_0\cdots x_k$ of vertices of $G$
such that each pair $x_{i-1},x_i$ span an edge.
Sequences of length $k=0$ are allowed.
The \emph{path algebra} of $G$ is the $\Z$-algebra
with basis the edge paths in $G$, and multiplication
is given by concatenation, where possible:
\begin{equation}\label{equation-concatenation}
	(x_0\cdots x_k)\cdot(y_0\cdots y_l)
	=
	\left\{\begin{array}{ll}
		x_0\cdots x_k y_1\cdots y_l
		& \text{if }x_k=y_0
		\\
		0
		&\text{otherwise}
	\end{array}\right.
\end{equation}
Thus the path algebra of $G$ is the path algebra of the
quiver obtained from $G$ by doubling each edge to give
two oriented edges, one in each direction.
(See~\cite[section~1]{CrawleyBoevey} for the definition of quivers
and their path algebras. We have defined our path algebras
over $\Z$, but it seems that path algebras of quivers
are usually defined as $k$-algebras for a field $k$.)
We make the path algebra into a graded algebra using path-length.
\end{definition}

\begin{theorem}[The diagonal part of the magnitude cohomology]
\label{theorem-diagonal}
Let $G$ be a finite graph.  Then the 
diagonal part of $\MH_\ast^\ast(G)$,
by which we mean the graded subring consisting of the groups $\MH_k^k(G)$,
is isomorphic to the quotient of the path algebra of $G$
by the relations
\begin{equation}\label{equation-diagonalrelations}
	\sum_{y\colon x\prec y\prec z} xyz=0
\end{equation}
for each pair of vertices $x,z$ with $d(x,z)=2$.
The symbol $x\prec y\prec z$ indicates that the sum is taken
over all $y$ for which $d(x,y)=d(y,z)=1$.
\end{theorem}

A diagonal graph has \emph{torsion-free} magnitude homology 
concentrated in the groups $\MH_{k,k}(G)$.
(See the proof of Proposition~7.2 of~\cite{HepworthWillerton}.)
The universal coefficient theorem of Remark~\ref{remark-universal}
then guarantees that the
magnitude cohomology of $G$ is torsion-free and concentrated
in the groups $\MH_k^k(G)$.  Thus we obtain the following.

\begin{corollary}\label{corollary-diagonal}
If $G$ is a diagonal graph, then $\MH_\ast^\ast(G)$
is exactly isomorphic to the quotient of the path
algebra described in Theorem~\ref{theorem-diagonal}.
\end{corollary}

\begin{example}[Trees]
Let $T$ be a finite tree with $n$ vertices.  
Then $T$ is diagonal, and its magnitude homology
was computed in~\cite[Corollary~6.8]{HepworthWillerton}.
If $x,z$ are vertices of $T$
with $d(x,z)=2$, then there is a unique vertex $y$ such that
$d(x,y)=d(y,z)=1$, and it follows that $xy\cdot yz=0$ 
in $\MH^2_2(T)$.
The only edge paths that are not rendered $0$ by this
relation are the ones of the form
\[	
	\underbrace{abab\cdots}_{(k+1)\text{ terms}}
\]
for $k\geqslant 0$ and $d(a,b)=1$, 
and these elements form a basis of $\MH_\ast^\ast(T)$.
This is dual to the description of $\MH_{\ast,\ast}(T)$
given in~\cite{HepworthWillerton}.
\end{example}

\begin{example}[Complete graphs]
\label{example-complete}
In Example~2.5 of~\cite{HepworthWillerton} it was shown that 
the complete graph $K_n$ on $n$ vertices is diagonal.
We may therefore apply Corollary~\ref{corollary-diagonal}.
Any sequence of vertices in $K_n$ is an edge path,
and there are no pairs $x,y$ with $d(x,y)=2$,
so $\MH^\ast_\ast(K_n)$ is precisely the algebra
with basis given by all finite sequences of vertices of $K_n$,
with product given by concatenation~\eqref{equation-concatenation}.
\end{example}

\begin{example}[Complete bipartite graphs]
The complete bipartite graph $K$ on two nonempty sets $X$ and $Y$
is the join of the graphs with $X$ and $Y$ as vertex sets and no edges.
It is diagonal by Theorem~7.5 of \cite{HepworthWillerton}, and
Corollary~\ref{corollary-diagonal} therefore applies.
The path algebra of $K$ has basis given by the finite sequences in $X\sqcup Y$
alternating between elements of $X$ and elements of $Y$,
and $\MH^\ast_\ast(K)$ is the quotient of this by the relations
\[
	\sum_{y\in Y} x_1yx_2=0,
	\qquad
	\sum_{x\in X}y_1xy_2=0
\]
for every pair of distinct elements $x_1,x_2\in X$ 
and $y_1,y_2\in Y$.
\end{example}

\begin{example}[The icosahedral graph]
Let $G$ denote the graph obtained by taking 
the $1$-skeleton of the icosahedron.
Theorem~4.5 of~\cite{Gu} shows that $G$ is diagonal,
so that we may apply Corollary~\ref{corollary-diagonal}.
The result is that $\MH^\ast_\ast(G)$
is the path algebra of $G$ modulo the ideal generated by the `diamond moves'
$xyz = -xy'z$ whenever $x,y,y',z$ form a diamond whose
points are $x$ and $z$:
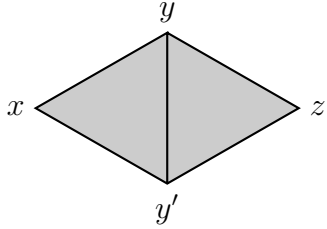
\[\begin{tikzpicture}
	\draw[thick,fill=white!80!black] 
		(0,1) --(1.73,0) -- (0,-1) -- (-1.73,0) -- (0,1) -- (0,-1);
	\node[above] at (0,1) {$y$};
	\node[below] at (0,-1) {$y'$};
	\node[left] at (-1.73,0) {$x$};
	\node[right] at (1.73,0) {$z$};
\end{tikzpicture}\]
\end{example}

\begin{proof}[Proof of Theorem~\ref{theorem-diagonal}]
The magnitude chain group $\MC_{k,\ell}(G)$ has basis
given by the $k$-simplices $(x_0,\ldots,x_k)$ of length $\ell$.
We equip the magnitude cochain group $\MC_\ell^k(G)$
with the dual basis, denoting the dual to $(x_0,\ldots,x_k)$
by $(x_0,\ldots,x_k)^\ast$.
Restricting our attention to the following part of the magnitude
cochain complex,
\[
	\MC_k^{k-1}(G)\xrightarrow{\ \partial^\ast\ }
	\MC_k^k(G)\xrightarrow{\ \partial^\ast\ }
	\MC_k^{k+1}(G)
\]
we see that $\MC_k^{k+1}(G)=0$ because any simplex of 
degree $(k+1)$ has length $\ell\geqslant (k+1)$, 
that $\MC_k^k(G)$ has basis given by the
$(x_0,\ldots,x_k)^\ast$ in which each pair
$x_{j-1},x_j$ spans an edge,
and that $\MC_{k}^{k-1}(G)$ has basis given by
the $(x_0,\ldots,x_{i-1},x_{i+1},\ldots,x_k)^\ast$
where each $x_{j-1},x_j$ is an edge and where $d(x_{i-1},x_{i+1})=2$.
One can check that the boundary map is determined by the rule
\begin{equation}\label{equation-relations}
	\partial^\ast(x_0,\ldots,x_{i-1},x_{i+1},\ldots, x_k)^\ast
	=
	(-1)^i
	\sum_{x_i\colon x_{i-1}\prec x_i\prec x_{i+1}}(x_0,\ldots, x_k)^\ast
\end{equation}
where, again, the symbol $x_{i-1}\prec x_i\prec x_{i+1}$
indicates that $d(x_{i-1},x_i)=d(x_i,x_{i+1})=1$.
Thus $\MH_k^k(G)$ is the quotient of the $\Z$-module
with basis the dual simplices $(x_0,\ldots,x_k)^\ast$ in which each
$x_{j-1},x_j$ is an edge, by the right-hand-sides of the
equations~\eqref{equation-relations}.

Now, we let $A_\ast$ denote the quotient of the path algebra
of $G$ by the relations~\eqref{equation-diagonalrelations}.
Then $A_k$ is the $\Z$-module with basis the edge paths $x_0\cdots x_k$,
modulo, for each sequence $x_0,\ldots,x_{i-1},x_{i+1}\ldots,x_k$
in which consecutive entries are edges except that $d(x_{i-1},x_{i+1})=2$,
the relation
$\sum_{x_i\colon x_{i-1}\prec x_i\prec x_{i+1}}x_0\cdots x_k$.
Thus, the map $A_k\to\MH_k^k(G)$, $x_0\cdots x_k\mapsto [(x_0,\ldots,x_k)^\ast]$
is an isomorphism of $\Z$-modules.
It remains to show that it is a map of graded rings,
but that is evident from the formula
\[
	(x_0,\ldots ,x_k)^\ast\cutprod(y_0,\ldots, y_l)^\ast
	=
	\left\{\begin{array}{ll}
		(x_0,\ldots, x_k, y_1,\ldots, y_l)^\ast
		& \text{if }x_k=y_0
		\\
		0
		&\text{otherwise}
	\end{array}\right.
\]
which is easily verified.
\end{proof}

\section{The magnitude cohomology of odd cyclic graphs}
\label{section-cyclic}

In this section we give an extended example: 
the magnitude cohomology ring of the cyclic graph $C_n$ 
with an odd number of vertices $n=(2m+1)\geqslant 5$.
The case of $C_1=K_1$ and $C_3=K_3$ is 
covered in Example~\ref{example-complete} above.  
Our computation of the magnitude {cohomology ring} is based on
Gu's computation of the magnitude {homology groups} 
given in section~4.4 of~\cite{Gu}.

\begin{theorem}\label{theorem-cyclic}
Let $n=2m+1$ where $m\geqslant 2$,
and let $C_n$ denote the cyclic graph on $n$ vertices.
Then the magnitude cohomology ring $\MH^\ast_\ast(C_n)$ 
is the bigraded associative ring with the following presentation.
The generators are:
\begin{itemize}
	\item
	$e_x\in\MH_0^0(C_n)$ for vertices $x$ of $C_n$.
	\item
	$\uu_{xy}\in\MH_1^1(C_n)$ for oriented edges $xy$ of $C_n$.
	\item
	$\vv_{xz}\in\MH^2_{m+1}(C_n)$ for ordered pairs $x,z$ with $d(x,z)=m$.
\end{itemize}
And the relations are:
\begin{itemize}
	\item
	$e_x^2=e_x$ for every vertex $x$.
	\item
	$e_xe_y=0$ for distinct vertices $x,y$.
	\item
	$\uu_{xy}=e_x\uu_{xy}=\uu_{xy}e_y$ for every oriented edge $xy$.
	\item
	$\vv_{xz}=e_x\vv_{xz}=\vv_{xz}e_z$ for every $x,z$ with $d(x,z)=m$.
	\item
	$\uu_{xy}\uu_{yz}=0$ if $xy$ and $yz$ are oriented edges with $x\neq z$.
	\item
	$\uu_{wx}\vv_{xz}=\vv_{wy}\uu_{yz}$ for every $w,x,y,z$
	in cyclic order with 
	$d(w,x)=1$, $d(x,y)=m$, $d(y,z)=1$.
	
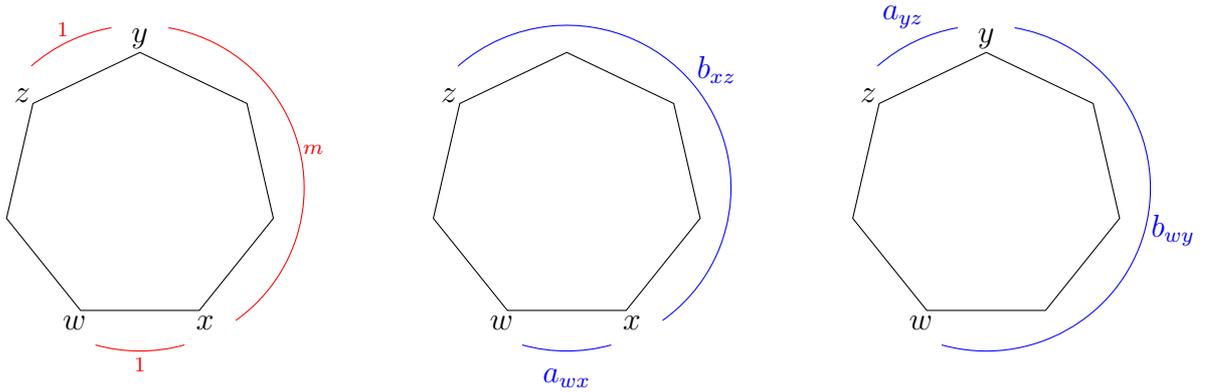
\[\begin{tikzpicture}[scale=0.18,baseline=0]
		\path[draw] (0+90:10)
			-- (360/7+90:10)
			-- (2*360/7+90:10)
			-- (3*360/7+90:10)
			-- (4*360/7+90:10)
			-- (5*360/7+90:10)
			-- (6*360/7+90:10)
			-- cycle;
		\node() at (3*360/7+90:11) {$w$};
		\node() at (4*360/7+90:11) {$x$};
		\node() at (7*360/7+90:11) {$y$};
		\node() at (8*360/7+90:11) {$z$};
		\path[draw,red] 	
			(3*360/7+90+10:12) arc (3*360/7+90+10:4*360/7+90-10:12);
		\path[draw,red] 	
			(4*360/7+90+10:12) arc (4*360/7+90+10:7*360/7+90-10:12);
		\path[draw,red] 	
			(7*360/7+90+10:12) arc (7*360/7+90+10:8*360/7+90-10:12);
		\node[red]() at (3.5*360/7+90:13) {$\scriptstyle 1$};
		\node[red]() at (5.5*360/7+90:13) {$\scriptstyle m$};
		\node[red]() at (7.5*360/7+90:13) {$\scriptstyle 1$};
	\end{tikzpicture}
	\qquad\quad
	\begin{tikzpicture}[scale=0.18,baseline=0]
		\path[draw] (0+90:10)
			-- (360/7+90:10)
			-- (2*360/7+90:10)
			-- (3*360/7+90:10)
			-- (4*360/7+90:10)
			-- (5*360/7+90:10)
			-- (6*360/7+90:10)
			-- cycle;
		\node() at (3*360/7+90:11) {$w$};
		\node() at (4*360/7+90:11) {$x$};
		\node() at (8*360/7+90:11) {$z$};
		\path[draw,blue] 	
			(3*360/7+90+10:12) arc (3*360/7+90+10:4*360/7+90-10:12);
		\path[draw,blue] 	
			(4*360/7+90+10:12) arc (4*360/7+90+10:8*360/7+90-10:12);
		\node[blue]() at (3.5*360/7+90:14) {$\uu_{wx}$};
		\node[blue]() at (6*360/7+90:14) {$\vv_{xz}$};
	\end{tikzpicture}
	\qquad\quad
	\begin{tikzpicture}[scale=0.18,baseline=0]
		\path[draw] (0+90:10)
			-- (360/7+90:10)
			-- (2*360/7+90:10)
			-- (3*360/7+90:10)
			-- (4*360/7+90:10)
			-- (5*360/7+90:10)
			-- (6*360/7+90:10)
			-- cycle;
		\node() at (3*360/7+90:11) {$w$};
		\node() at (7*360/7+90:11) {$y$};
		\node() at (8*360/7+90:11) {$z$};
		\path[draw,blue] 	
			(3*360/7+90+10:12) arc (3*360/7+90+10:7*360/7+90-10:12);
		\path[draw,blue] 	
			(7*360/7+90+10:12) arc (7*360/7+90+10:8*360/7+90-10:12);
		\node[blue]() at (5*360/7+90:14) {$\vv_{wy}$};
		\node[blue]() at (7.5*360/7+90:14) {$\uu_{yz}$};
	\end{tikzpicture}\]
\end{itemize}
\end{theorem}

In order to prove the theorem we start by establishing notation.
Fix a `clockwise' direction on the vertices of $C_n$.

\begin{definition}[Codes and admissible simplices]
\hfill
\begin{itemize}
	\item
	Given a vertex $x$ of $C_n$ and $i\in\{-m,\ldots,0,\ldots,m\}$,
	we let $x+i$ denote the vertex obtained by moving $|i|$ places from $x$,
	clockwise if $i>0$, and anticlockwise if $i<0$.

	\item
	Given a simplex $\bfx=(x_0,\ldots,x_k)$,
	we obtain a sequence $\bfi=(i_1,\ldots, i_k)$
	with entries in $\{-m,\ldots,-1,1,\ldots,m\}$ 
	defined by $x_j = x_{j-1} +i_j$.
	We call $\bfi$ the \emph{code} of $\bfx$.
	Note that the code of a $0$-simplex $(x_0)$ is 
	the empty tuple $()$.

	\item
	A code is called \emph{admissible} if, after dividing it into
	the maximal subsequences whose entries all have the same sign,
	the subsequences all have one of the following forms for some
	$j\geqslant 0$.
	\[
		\underbrace{(1,m,1,m,\ldots)}_{j\text{ entries}}
		\qquad
		\underbrace{(-1,-m,1,-m\ldots)}_{j\text{ entries}}
	\]
	Note that while these sequences begin with $\pm 1$, they
	can end with $\pm 1$ or $\pm m$, according to whether
	$j$ is odd or even.

	\item
	A simplex is \emph{admissible} if its code is admissible.
\end{itemize}
\end{definition}

\begin{example}
The codes $(1,m,1,-1,1,m)$ and $(1,-1,1,-1)$ are admissible,
but $(m,1)$ and $(-1,-m,1,1,m)$ are not.
\end{example}

\begin{note}
Let $\bfx=(x_0,\ldots,x_k)$ be an admissible simplex with code
$\bfi=(i_1,\ldots,i_k)$.  Then $\bfx$ has degree $k$
and length $\ell=\sum_{j=1}^k |i_j|$.
\end{note}

\begin{theorem}[{Gu~\cite{Gu}}]
\label{theorem-gu}
The admissible simplices $\bfx$ 
are cycles in $\MC_{\ast,\ast}(C_n)$,
and their homology classes $[\bfx]$ form a basis for $\MH_{\ast,\ast}(C_n)$.
\end{theorem}

This theorem will be crucial for our computations.
It has been stated for our own purposes,
and does not appear explicitly in~\cite{Gu}.  However,
it can easily be extracted from the proof of~\cite[Theorem~4.6]{Gu},
whose `unmatched simplices' are exactly our admissible simplices,
and whose final paragraph states that these unmatched simplices
form a basis for the homology.

\begin{example}\label{admissible-low-degrees-example}
Let us explore this description in degrees $0,1,2$.
\begin{itemize}
	\item
	The $0$-simplices $(x_0)$ are all admissible.

	\item
	In degree $1$ the admissible codes are $(1)$
	and $(-1)$, so the admissible sequences are $(x,y)$ 
	for $x,y$ adjacent, or in other words the oriented edges.

	\item
	In degree $2$ the admissible simplices are:
	\begin{itemize}
		\item
		$(x,y,x)$, one for each oriented edge $(x,y)$.
		They have length $2$.
		The corresponding codes are $(1,-1)$ and $(-1,1)$.
		\item
		$(x,y,z)$, one for each ordered pair $(x,z)$
		with $d(x,z)=m$, where $y$ is determined by the conditions 
		$d(x,y)=1$ and $d(y,z)=m$.  They have length $(m+1)$.
		The corresponding codes are $(1,m)$ and $(-1,-m)$.
	\end{itemize}
\end{itemize}
\end{example}

\begin{definition}
\label{definition-euv}
The universal coefficient sequence of Remark~\ref{remark-universal}
gives an isomorphism
of $\MH_\ast^\ast(C_n)$ with the dual of $\MH_{\ast,\ast}(C_n)$,
and so we obtain the basis of $\MH_\ast^\ast(C_n)$ dual to the one of
Theorem~\ref{theorem-gu}.
Using this dual basis we define elements of $\MH_\ast^\ast(C_n)$ as follows.
\begin{itemize}
	\item
	For each vertex $x$,
	$e_x\in\MH_0^0(C_n)$ is the dual to $[(x)]\in\MH_{0,0}(C_n)$.
	\item
	For each oriented edge $(x,y)$,
	$\uu_{xy}\in\MH^1_1(C_n)$ is the dual to $[(x,y)]\in\MH_{1,1}(C_n)$.
	\item
	For each pair $x,z$ with $d(x,z)=m$,
	$\vv_{xz}\in\MH^2_{m+1}(C_n)$ denotes the dual to 
	$[(x,y,z)]\in\MH_{2,m+1}(C_n)$.
	Here $y$ is determined by the conditions $d(x,y)=1$, $d(y,z)=m$
	as in Example~\ref{admissible-low-degrees-example}.
\end{itemize}
\end{definition}

\begin{lemma}
\label{lemma-relations}
The relations specified in Theorem~\ref{theorem-cyclic} hold.
\end{lemma}

\begin{proof}
One can check that, under the isomorphism of Theorem~\ref{theorem-diagonal},
$e_x$ corresponds to the path $x$ and $\uu_{xy}$ to the path $xy$,
and the theorem gives us the following relations:
\begin{itemize}
	\item
	$e_x^2=e_x$ for every vertex $x$.
	\item
	$e_xe_y=0$ for distinct vertices $x,y$.
	\item
	$\uu_{xy}=e_x\uu_{xy}=\uu_{xy}e_y$ for every oriented edge $xy$.
	\item
	$\uu_{xy}\uu_{yz}=0$ if $xy$ and $yz$ are oriented edges with $x\neq z$.
\end{itemize}
We now prove the relation:
\begin{itemize}
	\item
	$\vv_{xz}=e_x\vv_{xz}=\vv_{xz}e_z$ for every $x,z$ with $d(x,z)=m$.
\end{itemize}
Let $(a_0,a_1,a_2)$ be an admissible simplex. Then 
$
	\langle e_x\cdot \vv_{xy},[(a_0,a_1,a_2)]\rangle
	=
	\langle e_x, [(a_0)]\rangle\cdot \langle \vv_{xy},[(a_0,a_1,a_2)]\rangle
$
as one sees by choosing cocycles representing $e_x$ and $\vv_{xz}$.
For $(a_0,a_1,a_2)=(x,y,z)$ both factors evaluate to $1$,
and for any other choice of $(a_0,a_1,a_2)$ the second factor evaluates to $0$,
so that $e_x\cdot \vv_{xz}=\vv_{xz}$.  The other part of the relation is proved
similarly.
Now we prove the final relation:
\begin{itemize}
	\item
	$\uu_{wx}\vv_{xz}=\vv_{wy}\uu_{yz}$ for every $w,x,y,z$
	in cyclic order with 
	$d(w,x)=1$, $d(x,y)=m$, $d(y,z)=1$.
\end{itemize}
Let $(x_0,x_1,x_2,x_3)$ be an admissible simplex.
Then by choosing cocycles representing 
$\uu_{wx}$, $\vv_{xz}$, $\vv_{wy}$ and $\uu_{yz}$,
we see that
\[
	\langle \uu_{wx}\cutprod \vv_{xz},[(x_0,x_1,x_2,x_3)]\rangle
	=
	\langle \uu_{wx},[(x_0,x_1)]\rangle\cdot
	\langle \vv_{xz},[(x_1,x_2,x_3)]\rangle
\]
and
\[
	\langle \vv_{wy}\cutprod \uu_{yz},[(x_0,x_1,x_2,x_3)]\rangle
	=
	\langle \vv_{wy},[(x_0,x_1,x_2)]\rangle\cdot
	\langle \uu_{yz},[(x_2,x_3)]\rangle.
\]
Notice that in each case the right hand side vanishes unless
$x_0=w$, $x_3=z$, and $\ell(x_0,x_1,x_2,x_3)=(m+2)$.
The only admissible $3$-simplex with these properties is
$(x_0,x_1,x_2,x_3)=(w,x,y,z)$, and so it suffices to show
that the two right hand sides displayed coincide in this case.
This follows from the fact that $\langle \uu_{wx},[(w,x)]\rangle$, 
$\langle \uu_{yz},[(y,z)]\rangle$, $\langle \vv_{wy},[(w,x,y)]\rangle$
and $\langle \vv_{xz},[(x,y,z)]\rangle$ are all equal to $1$.
This is by definition in the first three cases.
In the final case, 
we note that $(x,y,z)$ has code 
$(m,1)$, but that if we let $y'$ denote the vertex with
$d(x,y')=1$ and $d(y',z)=m$, then 
$\partial(x,y',y,z)=-(x,y,z)+(x,y',z)$ so that $[(x,y,z)]=[(x,y',z)]$
and consequently $\langle \vv_{xz},[(x,y,z)]\rangle
= \langle \vv_{xz},[(x,y',z)]\rangle=1$.
So both sides of our relation coincide when evaluated on
any basis element of $\MH_{3,(m+2)}(C_n)$, and this completes
the proof.
\end{proof}

\begin{definition}[Monomials from admissible tuples]
\label{definition-monomials}
Suppose given an admissible simplex $\bfx$
with code $()$, $(1)$, $(-1)$, $(1,m)$ or $(-1,-m)$.
Then we define $p_\bfx\in\MH^\ast_\ast(C_n)$ to be the class
$e_{x_0}$, $\uu_{x_0x_1}$ or $\vv_{x_0x_2}$ dual to 
$[\bfx]$.
More generally, if $\bfx=(x_0,\ldots,x_k)$ is 
an admissible simplex with $k>0$,
then there is a unique way to decompose $\bfx$ into `pieces'
\[
	\bfx_1=(x_{i_0},\ldots,x_{i_1}),
	\ 
	\bfx_2=(x_{i_1},\ldots,x_{i_2}),
	\ 
	\ldots,
	\ 
	\bfx_r=(x_{i_{r-1}},\ldots,x_{i_r}),
\]
with $i_0=0$ and $i_r=k$, such that each piece has code
$(1)$, $(-1)$, $(1,m)$ or $(-1,-m)$, 
and in this case we define $p_\bfx = p_{\bfx_1}\cdots p_{\bfx_r}$.
\end{definition}

\begin{example}
An admissible simplex $\bfx=(x_0,x_1,x_2,x_3,x_2,x_3,x_1)$
with code $(1,m,1,-1,1,m)$ breaks into pieces
\[
	(x_0,x_1,x_2),
	\quad
	(x_2,x_3),
	\quad
	(x_3,x_2),
	\quad
	(x_2,x_3,x_1)
\]
with codes
\[
	(1,m),
	\quad
	(1),
	\quad
	(-1),
	\quad
	(1,m)
\]
respectively, and so the corresponding monomial is 
$p_\bfx = \vv_{x_0x_2}\uu_{x_2x_3}\uu_{x_3x_2}\vv_{x_2x_1}$.
\end{example}

\begin{lemma}
\label{lemma-monomials}
The relations of Theorem~\ref{theorem-cyclic} imply that
every monomial in the generators of Theorem~\ref{theorem-cyclic}
is either $0$,
or has the form $p_\bfx$ for some admissible simplex $\bfx$.
\end{lemma}

\begin{proof}
Lemma~\ref{lemma-relations} shows that the given relations
among the generators hold.  
Using the relations involving the $e_x$, we may ensure that
our monomial is either $0$, or a single $e_x$,
or that the monomial consists entirely of $\uu$'s and $\vv$'s.
Using the same relations again, we may ensure that the second subscript
of each term always coincides with the first subscript of the next term,
otherwise we obtain $0$ once more.
Let us say that $\uu_{xy}$ is \emph{clockwise} if $x,y$ are in clockwise order,
and \emph{anticlockwise} otherwise.  And let us say that
$\vv_{xz}$ is \emph{clockwise} if $x,z$ are in \emph{anti}clockwise order,
and \emph{anticlockwise} otherwise.
Now we factor our monomial $p$ into the maximal factors
$p_1,\ldots,p_r$ where each $p_i$ 
has entries that are all either clockwise or anticlockwise.
Our `clockwise' conventions ensure that in each $p_i$, 
any instance of an $\uu$ term preceding a $\vv$
term is an instance of the left-hand-side of the final relation
of Theorem~\ref{theorem-cyclic}.
We may therefore use that final relation to ensure that any
$\uu$ terms occur after any $\vv$ terms.
If we find more than one $\uu$ term, then $p_i=0$
and so $p=0$.
Otherwise, each $p_i$ now has form $p_{\bfx_i}$ for an appropriate
$\bfx_i$, and consequently $p=p_\bfx$ where
where $\bfx$ is obtained by combining the $\bfx_i$.
\end{proof}

\begin{lemma}
\label{lemma-kronecker}
Let $\bfx$ and $\bfy$ be admissible simplices.
Then
\[
	\langle p_\bfx,[\bfy]\rangle
	=
	\left\{
	\begin{array}{ll}
		1 & \text{if }\bfx=\bfy
		\\
		0 & \text{if }\bfx\neq\bfy
	\end{array}
	\right.
\]
The $p_\bfx$ for $\bfx$ admissible form a basis
of $\MH^\ast_\ast(C_n)$.
\end{lemma}

\begin{proof}
Suppose $\bfx=\bfy$.
Then, following the definition of the product
and the construction of $p_\bfx$,
we find that $\langle p_\bfx,[\bfx]\rangle
=\prod \langle p_{\bfx_i},[\bfx_i]\rangle$
where $\bfx_1,\ldots,\bfx_r$ is the decomposition
of $\bfx$ into `pieces' as in Definition~\ref{definition-monomials}.  
And by definition, 
each $\langle p_{\bfx_i},[\bfx_i]\rangle$ is equal to $1$.

Suppose now that $\langle p_\bfx,[\bfy]\rangle\neq 0$.
We will show that $\bfx=\bfy$.
Decompose $\bfx$ into $\bfx_1,\ldots,\bfx_r$
as in Definition~\ref{definition-monomials},
and decompose $\bfy$ {in parallel} with $\bfx$, 
so that if the pieces for $\bfx$ are
\[
	\bfx_1=(x_{i_0},\ldots,x_{i_1}),
	\ 
	\bfx_2=(x_{i_1},\ldots,x_{i_2}),
	\ 
	\ldots
	\ 
	\bfx_r=(x_{i_{r-1}},\ldots,x_{i_r}),
\]
then those for $y$ are 
\[
	\bfy_1=(y_{i_0},\ldots,y_{i_1}),
	\ 
	\bfy_2=(y_{i_1},\ldots,y_{i_2}),
	\ 
	\ldots
	\ 
	\bfy_r=(y_{i_{r-1}},\ldots,y_{i_r}).
\]
Each of the $\bfy_i$ is still a cycle,
and $\langle p_\bfx,[\bfy]\rangle
=\prod \langle p_{\bfx_i},[\bfy_i]\rangle$,
so that each $\langle p_{\bfx_j},[\bfy_j]\rangle$
must be nonzero, and in particular $x_{i_{j-1}}=y_{i_{j-1}}$
and $x_{i_{j}}=y_{i_{j}}$.
If follows that $x_{i_j}=y_{i_j}$ for $j=0,\ldots,r$.
The only way that $\bfx$ and $\bfy$ can now differ
is that $\bfx$ and $\bfy$ may now have pieces
$\bfx_j=(x_{i_{j-1}},x,x_{i_j})$ and
$\bfy_j=(x_{i_{j-1}},y,x_{i_j})$ 
with codes $(1,m)$ and $(m,1)$ respectively,
or with codes $(-1,-m)$ and $(-m,-1)$ respectively.
Let us suppose it is the positive case.
Consider the first instance of such a difference.
Since $\bfy$ is admissible, the term preceding
$(m,1)$ in the code of $\bfy$ must be $1$. 
The same must therefore be true of $\bfx$,
so that the code of $\bfx$ contains a subsequence $(1,1,m)$.
This is a contradiction, so that $\bfx$ and $\bfy$ must coincide.
\end{proof}

\begin{proof}[Proof of Theorem~\ref{theorem-cyclic}]
Let $A_\ast^\ast$ denote the bigraded ring determined by the presentation
in Theorem~\ref{theorem-cyclic}.
Definition~\ref{definition-euv} and Lemma~\ref{lemma-monomials}
determine a well-defined homomorphism 
$h\colon A_\ast^\ast\to \MH_\ast^\ast(C_n)$.
Definition~\ref{definition-monomials} and Lemma~\ref{lemma-monomials}
determine a spanning set (the $p_\bfx$ for $\bfx$ admissible)
for $A_\ast^\ast$, and Lemma~\ref{lemma-kronecker} shows that $h$
sends this spanning set into a basis of $\MH^\ast_\ast(C_n)$.
It follows that $h$ is an isomorphism.
\end{proof}

\end{document}